\documentclass[12pt]{article}
\usepackage{amssymb,amsmath, amsthm}
\usepackage{color,xcolor}
\newtheorem{lemma}{Lemma}[section]
\newtheorem{proposition}[lemma]{Proposition}
\newtheorem{theorem}[lemma]{Theorem}

\newtheorem{definition}[lemma]{Definition}

\allowdisplaybreaks

\makeatletter
\renewcommand*\env@matrix[1][*\c@MaxMatrixCols c]{%
  \hskip -\arraycolsep
  \let\@ifnextchar\new@ifnextchar
  \array{#1}}
\makeatother

\begin{document}

\title{A hyperplane restriction theorem for holomorphic mappings and its application for the gap conjecture}

\author{Yun Gao\footnote{School of Mathematical Sciences, Shanghai Jiaotong University, Shanghai, People's Republic of China. \textbf{Email:}~gaoyunmath@sjtu.edu.cn}, Sui-Chung Ng\footnote{School of Mathematical Sciences, Shanghai Key Laboratory of PMMP, East China Normal University, Shanghai, People's Republic of China. \textbf{Email:}~scng@math.ecnu.edu.cn}}

\maketitle

\begin{abstract}
We established a hyperplane restriction theorem for the local holomorphic mappings between projective spaces, which is inspired by the corresponding theorem of Green for $\mathcal O_{\mathbb P^n}(d)$. Our theorem allows us to give the first proof for the existence of gaps (albeit smaller) at all levels for the rational proper maps between complex unit balls, conjectured by Huang-Ji-Yin. In addition, our proof does not distinguish  the unit balls from other generalized balls and thus it simultaneously demonstrates the same phenomenon for all generalized balls.
\end{abstract}


\section{Introduction}

The structure of the set of rational proper maps between complex unit balls is a very classical topic in Several Complex Variables. Among the many unsolved problems in this topic, there is the well-known \textit{gap phenomenon}, which will be recalled now. Fix an integer $n\geq 2$. For each $k\in\mathbb N^+$ such that $k(k+1)/2<n$, define the closed interval $\mathcal I_k:=[kn+1,(k+1)n-\frac{k(k+1)}{2}-1].$ The classical theorem of Faran~\cite{faran} amounts to saying that when $N\in\mathcal I_1=[n+1,2n-2]$, any local holomorphic map sending an open piece of $\partial\mathbb B^n$ to $\partial\mathbb B^N$ actually maps $\partial\mathbb B^n$ to a linear section $\partial\mathbb B^n\subset\partial\mathbb B^N$. In other words, there are no ``new'' maps when $N$ increases from $n$ to $2n-2$.
Then, it was discovered by Huang-Ji-Xu~\cite{hjx} that the same phenomenon holds for $N\in\mathcal I_2=[2n+1,3n-4]$ and later by Huang-Ji-Yin~\cite{HJY} for $N\in\mathcal I_3=[3n+1,4n-7]$. The \textit{Gap Conjecture}, formulated in~\cite{HJY2}, states that the gap phenomenon holds whenever $N\in\mathcal I_k$. 

In this article we are going to establish the existence of similar gaps for all levels \textit{at once} and also to demonstrate the gap phenomenon actually holds for \textit{all} generalized balls (whose definition will be recalled below).

\begin{theorem}\label{intro thm 2}
Let $k,n\in\mathbb N^+$ such that $n>k(k+1)$. For the local proper holomorphic maps between generalized balls, the gap phenomenon holds over the intervals
$$
\mathcal J_k:=[kn+k, (k+1)n-(k^2+1)].
$$
\end{theorem}

For Theorem~\ref{intro thm 2}, what we are going to prove is that when the target dimension is within $\mathcal J_k$, then the image must lie in a hyperplane (Theorem~\ref{gap thm}). Thus, in the case of ordinary unit balls, it can be directly interpreted as there are no new proper holomorphic maps within the gaps, as described above. However, we will see that formulating precisely the gap phenomenon for all generalized balls as ``no new maps within the gaps" needs a bit more work and this will be done in Section~\ref{proper map section}. There the theorem above will be stated more precisely as Theorem~\ref{gap thm 2}.
Note that although the interval $\mathcal J_k$ in our theorem is smaller than the $\mathcal I_k$ in the original Gap Conjecture, this is to be expected since our theorem holds for \textit{all generalized balls}. As a matter of fact, the lower bound for $\mathcal J_k$ is sharp in the present context, as will be demonstrated after Theorem~\ref{gap thm 2}. On the other hand, we do not know at this point whether the upper bound for $\mathcal J_k$ is sharp for generalized balls.

Our proof for Theorem~\ref{intro thm 2} consists of two main ingredients: the orthogonality preserved by the relevant proper maps; and a hyperplane restriction theorem for holomorphic mappings. Regarding the study of orthogonality, it originated from an earlier work~\cite{GN}  of the authors. There we proposed a coordinate free approach to the rigidity problems related to real hyperquadrics on the projective space and generalized a number of well-known rigidity theorems by using rather simple arguments. The reader is referred to~\cite{GN} for the detail of this approach using orthogonality, although we will briefly recall the basics wherever needed in this article.

To state our hyperplane restriction theorem, we first bring out the fact that every positive integer $A$ can be written as certain sums of binomial coefficients. Fix any $n\in\mathbb N^+$, there exist unique positive integers $a_n>a_{n-1}>\cdots>a_\delta$, where $\delta\geq 1$ and $a_j\geq j$ for every $j$, such that
$A=\binom{a_n}{n}+\cdots+\binom{a_\delta}{\delta}$. This is called the \textit{$n$-th Macaulay's representation} of $A$ and its existence and uniqueness can be proved by a greedy algorithm. These representations originally appeared in Macaulay's work of homogeneous ideals in polynomial rings~\cite{Ma}. Using the $n$-th Macaulay representation of $A$, we define the operation $A^{-<n>}:=\binom{a_n-1}{n-1}+\cdots+\binom{a_\delta-1}{\delta-1}$. In what follows, ``span" means the projective linear span.

\begin{theorem}\label{hyperplane thm}
Let $f:U\subset\mathbb P^n\rightarrow\mathbb P^M$ be a local holomorphic map such that $\dim(\mathrm{span}(f(U))\geq N$ for some positive integer $N\leq M$. Then, for a general hyperplane $H$ such that $H\cap U\neq\varnothing$, $\dim(\mathrm{span}(f(H\cap U))\geq N^{-<n>}$. 
\end{theorem}

The equality in the theorem can hold, for example, when $f$ is a rational map whose components are all the linearly independent monic monomials of a fixed degree. Our theorem is obtained from combining Green's hyperplane restriction theorem (Theorem~\ref{green}) with a new combinatorial identity (Lemma~\ref{binom}). It holds for any local holomorphic maps between projective spaces and we believe that it will find applications elsewhere.

We now briefly explain the idea behind our proof for the gap phenomenon. Suppose $f$ is a local proper map from an $n$-dimensional generalized ball to an $N$-dimensional generalized ball such that the image of $f$ is not contained in any hyperplane. We first use our hyperplane restriction theorem repeatedly to obtain dimension estimates  for the linear spans of the images of the linear subspaces in all dimensions. Since we know that $f$ is orthogonal, for any pair of orthogonal subspaces in the source projective space, their images under $f$ will span two orthogonal subspaces in the target projective space. If $N$ falls within any of the ``gaps'', i.e. the intervals $\mathcal J_k$ introduced earlier, by some amount of arithmetic we can show that the previously obtained dimension estimates will imply that there are two orthogonal subspaces in the target space whose sum of dimensions is at least $N$, which is impossible.


\section{Macaulay representation and a hyperplane restriction theorem for holomorphic mappings}

Every positive integer $A$ can be written as certain sums of binomial coefficients. Fix any $n\in\mathbb N^+$, there exist unique positive integers $a_n>a_{n-1}>\cdots>a_\delta$, where $\delta\geq 1$ and $a_j\geq j$ for every $j$, such that
$A=\binom{a_n}{n}+\cdots+\binom{a_\delta}{\delta}$. This is called the $n$-th Macaulay's representation of $A$. These representations naturally appeared in the works of Macaulay~\cite{Ma} and Green~\cite{Gr} on homogeneous ideals in polynomial rings. There are several operations pertaining to the Macaulay's representations, as follows. 
Let $A=\binom{a_n}{n}+\cdots+\binom{a_\delta}{\delta}$ be the $n$-th Macaulay's representation of $A$, define 
$$
A_{<n>}=\binom{a_n-1}{n}+\cdots+\binom{a_\delta-1}{\delta};\,\,\,\,\,\,\,\, 
A^{-<n>}=\binom{a_n-1}{n-1}+\cdots+\binom{a_\delta-1}{\delta-1}.
$$
Here, we employ the convention that $\binom{a}{b}=0$ whenever $a<b$ or $b=0$. The seemingly peculiar choices of notations for these operations are due the fact that $A_{<n>}$ was used by Green and $A^{<n>}:=\binom{a_n+1}{n+1}+\cdots+\binom{a_\delta+1}{\delta+1}$ was used by Macaulay and $A^{-<n>}$ is in some sense the opposite of $A^{<n>}$.

We will need the following Green's hyperplane restriction theorem, which  has already been used in~\cite{GLV} to study CR mappings between real hyperquadrics.

\begin{theorem}[\cite{Gr}]\label{green}
Let $W$ be a complex vector subspace of $H^0(\mathcal O_{\mathbb P^n}(d))$ of codimension $c$.
Let $W_H\subset H^0(\mathcal O_H(d))$ be the restriction of $W$ to a general hyperplane $H$ and $c_H$ be its codimension. Then, $c_H\leq c_{<d>}$.
\end{theorem}

We begin by stating a key combinatorial lemma related to these operations, which connects Green's hyperplane theorem and our Theorem~\ref{hyperplane thm}. To streamline the presentation, we defer its proof to the Section~\ref{last section}.
\begin{lemma}\label{binom}
Suppose $m,k\geq 1$ and $A,B\geq 0$. If $A+B=\binom{m+k}{k}-1$, then $$A^{-<m>}+B_{<k>}=\binom{m+k-1}{k}-1.$$
Here, we adopt the convention that $0^{-<m>}=0_{<k>}=0$.
\end{lemma}

We can now prove our hyperplane restriction theorem. 

\begin{proof}[Proof of Theorem~\ref{hyperplane thm}]
We first prove the theorem for rational maps. Let $f:\mathbb P^n\dasharrow\mathbb P^M$ be a rational map such that $\dim(\mathrm{span}(f(\mathbb P^n))\geq N$. Let $f=[f_0,\ldots,f_M]$, where $f_j\in\mathbb C[z_0,\ldots,z_n]$, $0\leq j\leq M$, are homogeneous polynomials of degree $d$ without non-constant common factors. 

Let $\mathcal H_{d,n}\subset\mathbb C[z_0,\ldots,z_n]$ be the vector subspace of homogeneous polynomials of degree $d$ and $W\subset\mathcal H_{d,n}$ be the subspace spanned by $f_0,\ldots,f_M$. 
Thus, we have 
$$
N+1\leq\dim(W)\leq\dim(\mathcal H_{d,n})=\binom{n+d}{d}.
$$ 
For a hyperplane $H\subset\mathbb P^n$, choose a set of homogeneous coordinates on $H$ and let $W_H\subset\mathcal H_{d,n-1}$ be the subspace spanned by restrictions of $f_0,\ldots,f_M$ on $H$. By Green's Theorem~\ref{green}, for a general hyperplane $H$, we have
$$
\dim(\mathcal H_{d,n-1})-\dim(W_H)\leq \left(\dim(\mathcal H_{d,n})-\dim(W)\right)_{<d>}
\leq \left(\dim(\mathcal H_{d,n})-N-1\right)_{<d>},
$$
in which the last inequality follows from the fact that $c_{<d>}\leq c'_{<d>}$ if $c\leq c'$. Thus,
\begin{equation}\label{dim wh}
\dim(W_H)
\geq\binom{n+d-1}{d}-\left(\binom{n+d}{d}-N-1\right)_{<d>}
\end{equation}
Now let $B=\binom{n+d}{d}-N-1$. Then $B\geq 0$ and 
$$
N+B=\binom{n+d}{d}-1.
$$
By Lemma~\ref{binom}, we have from~$(1)$
\begin{eqnarray*}
\dim(W_H)
&\geq&\binom{n+d-1}{d}-B_{<d>}\\
&=&\binom{n+d-1}{d}-\left(\binom{n+d-1}{d}-1-N^{-<n>}\right)\\
&=&N^{-<n>}+1.
\end{eqnarray*}
Thus $\dim(\mathrm{span}(f(H))\geq N^{-<n>}$ and we have proved the theorem for rational maps.

For the general case, if $f:U\subset\mathbb P^n\rightarrow\mathbb P^M$ is a local holomorphic map such that $\dim(\mathrm{span}(f(U))\geq N$, then for a sufficiently large $k$, the $k$-th order jet of $f$ can be represented by a rational map $f^\flat:\mathbb P^n\dasharrow\mathbb P^m$ (e.g. a truncated Taylor polynomial of $f$ at a point in $U$ after homogenization) such that $\dim(\mathrm{span}(f^\flat(\mathbb P^n))\geq N$. Since the restriction of $f^\flat$ to a hyperplane $H\subset\mathbb P^n$ represents the $k$-th order jet of $f|_{U\cap H}$, we see from the proven case of rational maps that for a general hyperplane $H$, we have 
$
\dim(\mathrm{span}(f(H\cap U))\geq \dim(\mathrm{span}(f^\flat(H))\geq N^{-<n>}.
$
\end{proof}

\noindent\textbf{Remark.} The equality in Theorem~\ref{hyperplane thm} can hold since the equality can hold in Green's theorem~\cite{Gr}. One can also see directly that the equality holds when $F$ is the rational map whose components are all the linearly independent monic monomials of a fixed degree in $\mathbb C[z_0,\ldots,z_n]$.   

Sometimes it is convenient to use the following counterpart of Theorem~\ref{hyperplane thm} and we will also elaborate a couple of special cases which are very useful  for the study of CR mappings between real hyperquadrics (e.g. see~\cite{GN}).

\begin{theorem}\label{faran type}
Let $g:U\subset\mathbb P^m\rightarrow\mathbb P^{m'}$ be a local holomorphic map and $\ell\in\mathbb N^+$ such that $\ell\leq m-1$. If $g$ maps $\ell$-planes to $\ell'$-planes, then it maps $(\ell+1)$-planes to $((\ell'+1)^{<\ell>}-1)$-planes. In particular,

$(i)$ if $\ell'\leq\ell-1$, then the image of $g$ is contained in an $\ell'$-plane;

$(ii)$ if $\ell\leq \ell'\leq 2\ell-1$, then $g$ maps $(\ell+k)$-planes to $(\ell'+k)$-planes for $k\geq 0$;

\end{theorem}
\begin{proof}
Suppose on the contrary the image of a general $(\ell+1)$-plane under $g$ is not contained in any $((\ell'+1)^{<\ell>}-1)$-plane. Since $((\ell'+1)^{<\ell>})^{-<\ell+1>}=\ell'+1$, Theorem~\ref{hyperplane thm} implies that the image of a general $\ell$-plane is not contained in any $\ell'$-plane.

If $\ell'\leq\ell-1$, then $\ell'+1=\binom{\ell}{\ell}+\binom{\ell-1}{\ell-1}+\cdots+\binom{\delta}{\delta}$ for some $\delta\geq 1$, so $(\ell'+1)^{<\ell>}-1=\ell'$. Therefore we deduce inductively that the image of $g$ is contained in an $\ell'$-plane.

If $\ell\leq \ell'\leq 2\ell-1$, then $\ell'+1=\binom{\ell+1}{\ell}+\binom{\ell-1}{\ell-1}+\binom{\ell-2}{\ell-2}+\cdots+\binom{\delta}{\delta}$ for some $\delta\geq 1$. Thus,
$(\ell'+1)^{<\ell>}-1=\ell'+1$ and so $g$ maps $(\ell+1)$-planes to $(\ell'+1)$-planes. Moreover, as $\ell+1\leq\ell'+1< 2(\ell+1)-1$, we can proceed inductively and the desired result follows.
\end{proof}

\noindent\textbf{Remark.} One can apply Theorem~\ref{faran type} repeatedly to get the following simple formula. Under the same hypotheses, if the $\ell$-th Macaulay's representation of $\ell'+1$ is $\binom{\lambda_\ell}{\ell}+\cdots+\binom{\lambda_\delta}{\delta}$, then for any $k\in\mathbb N^+$, $g$ maps every $(\ell+k)$-plane to some linear subspace of dimension 
$\binom{\lambda_\ell+k}{\ell+k}+\cdots+\binom{\lambda_\delta+k}{\delta+k}-1$.


\section{Gap phenomenon for local orthogonal maps}\label{gap}

We now recall some basics of \textit{local orthogonal maps} and the reader can see~\cite{GN} for more detail.

Let $r,s,t\in\mathbb N$ such that $r+s+t>0$. Denote by $\mathbb C^{r,s,t}$ the Euclidean space $\mathbb C^{r+s+t}$ equipped with the standard Hermitian bilinear form $\langle z, w\rangle_{r,s,t}$ of signature $(r,s,t)$, i.e.
$$
	\langle z, w\rangle_{r,s,t}
	=z_1\bar w_1+\cdots+z_r\bar w_r-z_{r+1}\bar w_{r+1}-\cdots - z_{r+s}\bar w_{r+s},
$$
where $z=(z_1,\ldots,z_{r+s+t})$ and $w=(w_1,\ldots,w_{r+s+t})$.
Define the indefinite norm  $\|z\|^2_{r,s,t}=\langle z, z\rangle_{r,s,t}$ and call any $z\in \mathbb C^{r,s,t}$ a \textit{positive point} if $\|z\|^2_{r,s,t}>0$; a \textit{negative point} if  $\|z\|^2_{r,s,t}<0$ and a \textit{null point} if  $\|z\|^2_{r,s,t}=0$. If $\langle z, w\rangle_{r,s,t}=0$, we say that $z$ is orthogonal to $w$ and write $z\perp w$. In addition, the \textit{orthogonal complement} of $z$ is defined as
$$z^{\perp}=\{w\in  \mathbb C^{r,s,t} \mid \langle z, w\rangle_{r,s,t}=0\}.$$

We denote by $\mathbb P^{r,s,t}:=\mathbb P\mathbb C^{r,s,t}$ the projectivization.  We write $\mathbb C^{r,s}$ and $\mathbb P^{r,s}$ instead of $\mathbb C^{r,s,0}$ and $\mathbb P^{r,s,0}$. 
On $\mathbb P^{r,s,t}$, even though the inner product is no longer defined, the notions of positive points, negative points and null points still make sense. Likewise, and more importantly, the orthogonality remains well defined. On $\mathbb P^{r,s,t}$, the set of positive points $\mathbb B^{r,s,t}\subset\mathbb P^{r,s,t}$ is called a \textit{generalized ball} since $\mathbb B^{1,s}$ is just the ordinary $s$-dimensional complex unit ball $\mathbb B^s$ embedded in $\mathbb P^s$. In addition, the boundary $\partial \mathbb B^{r,s,t}$ of $\mathbb B^{r,s,t}$ is simply the set of null points on $\mathbb P^{r,s,t}$.

\begin{definition}[\cite{GN}]\label{orthogonal def}
Let $U \subset \mathbb P^{r,s,t}$ be  a connected open set containing a null point. We call a holomorphic map $f: U\rightarrow\mathbb P^{r',s',t'}$ \textbf{orthogonal} if $f(p)\perp f(q)$ for any $p, q\in U$ such that $p\perp q$.
\end{definition}

\begin{proposition}\label{contradict}
Let $f:U\subset\mathbb P^{r,s,t}\rightarrow\mathbb P^{r',s'}$ be a local orthogonal map. Then, for every linear subspace $E\subset\mathbb P^{r,s,t}$ such that $E\cap U\neq\varnothing$ and $E^\perp\cap U\neq\varnothing$, 
$$
\dim(\mathrm{span}(f(E\cap U))+\dim(\mathrm{span}(f(E^\perp\cap U))\leq \dim(\mathbb P^{r',s'})-1.
$$
\end{proposition}
\begin{proof}
By orthogonality, $f(E^\perp\cap U)\subset (f(E\cap U))^\perp$ and since the Hermitian form on $\mathbb C^{r',s'}$ is non-degenerate, we have
$$
\dim(\mathrm{span}(f(E^\perp\cap U))+1\leq\dim(\mathbb C^{r',s'})-(\dim(\mathrm{span}(f(E\cap U))+1),
$$ and the desired result follows.
\end{proof}

Let $n,N\in\mathbb N$ such that $n+1\leq N<\binom{n+2}{2}=\binom{n+2}{n}$. By considering the $n$-th Macaulay representation of $N$, we deduce that $N$ is of the following form:
$$
N=N(n;a,b):=\binom{n+1}{n}+\cdots+\binom{n-a+1}{n-a}+ b
$$
for some integers $a,b\geq 0$ such that $b\leq n-a-1$. In fact, the $n$-th Macaulay's representation of $N(n;a,b)$ is
$$
N(n;a,0)=\binom{n+1}{n}+\cdots+\binom{n-a+1}{n-a}
$$
and for $b\geq 1$,
$$
N(n;a,b)=\binom{n+1}{n}+\cdots+\binom{n-a+1}{n-a}+ \binom{n-a-1}{n-a-1}+\cdots+\binom{n-a-b}{n-a-b}.
$$

\begin{lemma}\label{nab}
$\displaystyle
N(n;a,b)^{-<n>}=\left\{
\begin{matrix}[lcl]
N(n-1;a,b)&\mathrm{if}& n-a-b\geq 2;\\
N(n-1;a,b-1)&\mathrm{if}& n-a-b=1\,\,\mathrm{and}\,\,b\geq 1. 
\end{matrix}
\right.
$
\end{lemma}
\begin{proof}
It is an immediate consequence of the $n$-th Macaulay representation of $N(n;a,b)$ 
described above.
\end{proof}

\begin{proposition}\label{dim prop}
Let $a,b,n$ be non-negative integers. Let $g:U\subset\mathbb P^n\rightarrow\mathbb P^{N(n;a,b)}$ be a local holomorphic map whose image is not contained in a proper linear subspace. Let $D_m=\dim(\mathrm{span}(g(M\cap U)))$ for a general $m$-dimensional linear subspace $M$ intersecting $U$. Then,
$$
D_m\geq\left\{
\begin{matrix}[lcl]
N(m;a,b) &\,\,\mathrm{if}& a+b+1\leq m\leq n-1;\\
N(m;a,m-a-1) &\,\,\mathrm{if}& a+1\leq m\leq a+b 
\end{matrix}
\right.
$$
\end{proposition}

\begin{proof}
We will apply Theorem~\ref{hyperplane thm} and Lemma~\ref{nab} repeatedly with $m$ descending from $n-1$.

If $n-a-b\geq 2$, from the first line of Lemma~\ref{nab}, we get that $D_m\geq N(m;a,b)$ for $m=n-1,n-2,\ldots,a+b+1$. The proof is complete here if $b=0$. For $b\geq 1$, when we reach $m=a+b+1$, we have $N(m;a,b)=N(m;a,m-a-1)$. Thus, we deduce from the second line of Lemma~\ref{nab} that for $m=a+b,\dots, a+1$, we always have $D_m\geq N(m;a,m-a-1)$. 

If $n-a-b=1$, the condition for the first inequality never holds and so we just need to prove the second inequality. The argument is exactly the same as that in the previous paragraph by using the second line of Lemma~\ref{nab}. 
\end{proof}

\begin{theorem}\label{gap thm}
Let $f$ be a local orthogonal map from $\mathbb P^{r,s}$ to $\mathbb P^{r',s'}$ and $n:=\dim(\mathbb P^{r,s})$, $n':=\dim(\mathbb P^{r',s'})$. If there exists a non-negative integer $a$ such that
$$(a+1)(n+1)\leq n'\leq (a+2)n-(a^2+2a+2),$$
then the image of $f$ lies in a hyperplane of $\mathbb P^{r',s'}$
\end{theorem}
\begin{proof}
We will prove by contradiction. Suppose in the inequality in the hypotheses is satisfied for some $a$ and the image of $f$ is not contained in any hyperplane in $\mathbb P^{r',s'}$.  

Since 
$$
(a+1)(n+1)=\binom{n+1}{n}+\cdots+\binom{n-a+1}{n-a}+\dfrac{(a+1)a}{2}
$$ 
and
$$
(a+2)n-(a^2+2a+2)=\binom{n+1}{n}+\cdots+\binom{n-a+1}{n-a}+\left(n-\dfrac{a^2+5a+6}{2}\right),
$$
it follows that
$$
n'=\binom{n+1}{n}+\cdots+\binom{n-a+1}{n-a}+b=N(n;a,b)
$$
for some $b$ satisfying
\begin{equation}\label{ineq 1}
\dfrac{(a+1)a}{2}\leq b\leq n-\dfrac{a^2+5a+6}{2}.
\end{equation}

Let $n_1:=\left[\dfrac{n-1}{2}\right]$ and
$
n_2:=\left\{
\begin{matrix}n_1&\mathrm{ if\,\,} n \mathrm{\,\, is\,\, odd;}\\
n_1+1&\mathrm{ if\,\,} n \mathrm{\,\, is\,\, even.}
\end{matrix}\right.
$
Then $n_1+n_2+1=n$. Since $\langle\cdot,\cdot\rangle_{r,s}$ is non-degenerate, for an $n_1$-dimensional linear subspace in $\mathbb P^{r,s}$, its orthogonal complement is of dimension $n_2$ and conversely,  any $n_2$-dimensional linear subspace is the orthogonal complement of some $n_1$-dimensional linear subspace. Let $D_m$ be the dimension of the linear span of the image under $f$ of a general $m$-dimensional linear subspace intersecting the domain of definition of $f$. We are going to use Proposition~\ref{contradict} to reach a contradiction by showing that $D_{n_1}+D_{n_2}\geq n'$.

By~(\ref{ineq 1}), we have $n\geq a^2+3a+3$, thus
$$
n_1\geq \dfrac{n-2}{2}\geq\dfrac{a^2+3a+1}{2}\geq a+\dfrac{1}{2},
$$
from which we always have $n_1\geq a+1$ since $n_1$ is an integer.
We now consider the following two cases separately:
$$
\mathbf{ Case\,\, I:\,\,\,} b\leq n_1-a-1
\,\,\,\,\,\,\,\,\,\,\mathrm{and}\,\,\,\,\,\,\,\,\,\,
\mathbf{ Case\,\, II:\,\,\,}n_1-a\leq b.
$$

In \textbf{Case I}, since $$n-1\geq n_2\geq n_1\geq a+b+1,$$ so by Proposition~\ref{dim prop} and (\ref{ineq 1}), 
\begin{eqnarray*}
D_{n_1}+D_{n_2}&\geq& N(n_1;a,b)+N(n_2;a,b)\\
&=&\dfrac{a+1}{2}(2n_1+2-a)+b+\dfrac{a+1}{2}(2n_2+2-a)+b\\
&=&(a+1)(n+1-a)+2b\\
&=&(a+1)(n+1-\dfrac{a}{2})+b+\left(b-\dfrac{(a+1)a}{2}\right)\\
&\geq&(a+1)(n+1-\dfrac{a}{2})+b\\
&=&N(n;a,b)=n',
\end{eqnarray*}
which contradicts Proposition~\ref{contradict}.

In \textbf{Case II},  we have $a+1\leq n_1\leq a+b$ and thus by Proposition~\ref{dim prop}, 
$$
D_{n_1}\geq N(n_1;a;n_1-a-1)=(a+2)n_1-\dfrac{a^2+a}{2}
$$
and
$$
D_{n_2}\geq\left\{
\begin{matrix}
N(n_2;a;b) &\mathrm{\,\,if}& n_1=a+b &\textrm{and}& n_2=n_1+1;\\
N(n_2;a;n_2-a-1) &\mathrm{\,\,if}& n_1<a+b&\textrm{or}& n_2=n_1.
\end{matrix}
\right.
$$
Therefore,
{\small
$$
D_{n_1}+D_{n_2}\geq\left\{
\begin{matrix}
(a+2)(n_1+n_2+1)-(a^2+a)+b-n_2-1 &\mathrm{\,\,if}& n_1=a+b &\textrm{and}& n_2=n_1+1;\\
(a+2)(n_1+n_2)-(a^2+a) &\mathrm{\,\,if}& n_1<a+b&\textrm{or}& n_2=n_1,
\end{matrix}
\right.
$$
}
which simplifies to
$$
D_{n_1}+D_{n_2}\geq\left\{
\begin{matrix}
(a+2)n-(a^2+2a+2)&\mathrm{\,\,if}& n_1=a+b &\textrm{and}& n_2=n_1+1;\\
(a+2)n-(a^2+2a+2)&\mathrm{\,\,if}& n_1<a+b&\textrm{or}& n_2=n_1.
\end{matrix}
\right.
$$
Thus, we always have $D_{n_1}+D_{n_2}\geq n'$, which again contradicts Proposition~\ref{contradict}.
\end{proof}

\section{Proper maps between generalized balls}\label{proper map section}

We will now translate the results of the previous section to results for local proper holomorphic maps and formulate precisely the gap phenomenon for all generalized balls.

Let $V$ be a complex vector space equipped with a Hermitian inner product $H_V$ (possibly degenerate or indefinite) of signature $(r;s;t)$, where $\dim(V)=r+s+t$. Let $\mathbb PV$ be its projectivization. Similar to $\mathbb P^{r,s,t}$, the notion of positivity, negativity, nullity and orthogonality can be defined on $\mathbb PV$. In addition, any linear isometry $F:\mathbb C^{r,s,t}\rightarrow V$ induces a biholomorphic map $\tilde F:\mathbb P^{r,s,t}\rightarrow\mathbb PV$ preserving all these notions. Sufficient for our purpose, we  can simply identify any such projective space $\mathbb PV$ with $\mathbb P^{r,s,t}$ through any such biholomorphism and we write $\mathbb PV\cong\mathbb P^{r,s,t}$ for such identification.

Now let $H$ be a complex linear subspace in $\mathbb C^{r,s,t}$ and the restriction of  $\langle\cdot,\cdot\rangle_{r,s,t}$ on $H$ has the signature $(a; b; c)$. Obviously, we have $0\le a\le r$, $0\le b\le s$, $0\le c\le \min\{r-a,s-b\}+t$ and $a+b+c=\dim(H)$.  Then $\mathbb PH \cong \mathbb P^{a,b,c}$. We call $\mathbb P H$ an\textit{ $(a,b,c)$-subspace} of $\mathbb P^{r,s,t}$. We will often denote an $(a,b,c)$-subspace by $H^{a,b,c}$.

We use the following definition for local proper holomorphic maps between generalized balls:

\begin{definition}
A local holomorphic map $f:U\subset\mathbb P^{r,s,t}\rightarrow\mathbb P^{r',s',t'}$, defined on a connected open set $U$ such that $U\cap\partial \mathbb B^{r,s,t}\neq\varnothing$, is called a \textbf{local proper holomorphic map} from $\mathbb B^{r,s,t}$ to $\mathbb B^{r',s',t'}$ if $f(U\cap \mathbb B^{r,s,t})\subset \mathbb B^{r',s',t'}$ and $f(U\cap\partial \mathbb B^{r,s,t})\subset\partial \mathbb B^{r',s',t'}$.
\end{definition}

\begin{proposition}\label{equiv2}
By shrinking the domain of definition if necessary, a local proper holomorphic map from $\mathbb B^{r,s,t}$ to $\mathbb B^{r',s',t'}$ is a local orthogonal map from $\mathbb P^{r,s,t}$ to $\mathbb P^{r',s',t'}$.
\end{proposition}
\begin{proof}
The proof is essentially the same as the one used to prove that such a local proper holomorphic map respects the Segre varieties associated to the boundaries of the generalized balls. Note that in this case a Segre variety is just the orthogonal complement of a single point. (Readers may see~\cite{GN}, Proposition 2.5 therein, for the detail.)
\end{proof}

The reason why it is a bit more complicated to formulate the gap phenomenon for all generalized balls is that by adding ``canceling components'' one can trivially modify any given proper map $f$ between generalized balls to get a new map $f^\sharp$, whose target is a bigger generalized ball. The image of $f^\sharp$ lies in a hyperplane of the target generalized ball, but is not equivalent to $f$ under automorphisms. This is in contrast to the case where the target is an ordinary unit ball. To incorporate this, we will need a couple of definitions:

\begin{definition}[\cite{GN}] In Definition~\ref{orthogonal def}, a local orthogonal map from $\mathbb P^{r,s,t}$ to $\mathbb P^{r',s',t'}$ is called \textbf{null} if its image lies entirely in a null space of $\mathbb P^{r',s',t'}$. 
\end{definition}

\noindent\textbf{Remark.} If a local proper holomorphic map from $\mathbb B^{r,s,t}$ to $\mathbb B^{r',s',t'}$ is null, then it follows that its image is contained in $\partial\mathbb B^{r',s',t'}$. 

For any $(r',s')$-subspace $H^{r',s'}\subset\mathbb P^{R,S}$, by considering the corresponding orthogonal decomposition of $\mathbb C^{R,S}$, it is not difficult to see that there are two canonical projections $\pi$ and $\pi^\perp$ (as rational maps) from $\mathbb P^{R,S}$ to $H^{r',s'}$ and $(H^{r',s'})^\perp$ respectively.

\begin{definition}\label{null prolong}
Let $f$ be a local proper holomorphic map from $\mathbb B^{r,s}$ to $\mathbb B^{R,S}$. If there exists an $(r',s')$-subspace $H^{r',s'}\subset\mathbb P^{R,S}$, such that either 

$(i)$ the image of $f$ is contained in $H^{r',s'}$; or

$(ii)$ $\pi\circ f$ is a local proper holomorphic map from $\mathbb B^{r,s}$ to $\mathbb B^{R,S}\cap H^{r',s'}\cong \mathbb B^{r',s'}$ and $\pi^\perp\circ f$ is null,

then we say that $f$ is a {\it \textbf{null prolongation}} of $\pi\circ f$.
\end{definition}

\noindent\textbf{Example.} If $f=[f_1,\ldots,f_{r'+s'}]$ is a rational proper holomorphic map from $\mathbb B^{r,s}$ to $\mathbb B^{r',s'}$, where each $f_j\in\mathbb C[z_1,\ldots,z_{r+s}]$ is a degree-$d$ homogeneous polynomial, then for any homogeneous $\psi,\phi\in \mathbb C[z_1,\ldots,z_{r+s}]$ with $\deg(\phi)=\deg(\psi)+d$, the map $$F:=[\psi f_1,\ldots,\psi f_{r'},\phi,\psi f_{r'+1},\ldots,\psi f_{r'+s'},\phi],$$ which is locally proper from $\mathbb B^{r,s}$ to $\mathbb B^{r'+1,s'+1}$, is a null prolongation of $f$. 

On the other hand, a similar construction is not possible if we restrict ourselves to unit balls. Indeed, it is easy to see that for maps between unit balls, only $(i)$ can happen in Definition~\ref{null prolong}, i.e. a null prolongation can only be a local proper holomorphic map whose image lies in a smaller dimensional unit ball.

\begin{theorem}\label{gap thm 2}
Let $f$ be a local proper holomorphic map from $\mathbb B^{r,s}$ to $\mathbb B^{R,S}$ and $n:=\dim(\mathbb B^{r,s})$ and $N:=\dim(\mathbb B^{R,S})$. If there exists a positive integer $k$ such that
$$kn+k\leq N\leq (k+1)n-(k^2+1),$$
then there exists an $(r',s')$-subspace $H^{r',s'}\subset\mathbb P^{R,S}$, 
with $\dim(H^{r',s'})=kn+k-1$, such that $f$ is a null prolongation of some local proper holomorphic map from $\mathbb B^{r,s}$ to $\mathbb B^{R,S}\cap H^{r',s'}\cong \mathbb B^{r',s'}$.
\end{theorem}
\begin{proof}[Proof and remarks.]
In Theorem~\ref{gap thm}, by substituting $k:=a+1$, we see that whenever the hypotheses are satisfied, the image of $f$ is contained in a hyperplane $H\subset\mathbb P^{r,s}$. Write $H\cong\mathbb P^{r_1,s_1,t_1}$. Let $H^{r_1,s_1}\subset H$ be any $(r_1,s_1)$-subspace and $\pi:H\dasharrow H^{r_1,s_1}$, $\pi^\perp:H\dasharrow (H^{r_1,s_1})^\perp$ be the canonical projections. (Note that $(H^{r_1,s_1})^\perp\cong\mathbb P^{0,0,t_1}$ is a null space in $\mathbb P^{R,S}$.) Then, it follows that $f$ is a null prolongation of $f_1:=\pi\circ f$, where the latter is locally proper from $\mathbb B^{r,s}$ to $\mathbb B^{R,S}\cap H^{r_1,s_1}\cong\mathbb B^{r_1,s_1}$. The desired result then follows if we repeat the argument for a finite number of times, because if $f_1$ is a null prolongation of another local proper holomorphic map $f_2$ from $\mathbb B^{r,s}$ to $\mathbb B^{r_1,s_1}\cap H^{r_2,s_2}\cong\mathbb B^{r_2,s_2}$ for some $(r_2,s_2)$-subspace $H^{r_2,s_2}$, then $f$ is also a null prolongation of $f_2$.

For proper holomorphic maps from $\mathbb B^{1,n}\cong \mathbb B^n$ to $\mathbb B^{1,N}\cong\mathbb B^N$, Faran's result~\cite{faran} is essentially the statement that the conclusion of the theorem holds for $n+1\leq N\leq 2n-2$. Moreover, the same conclusion has been shown by Huang-Ji-Xu~\cite{hjx} to hold for $2n+1\leq N\leq 3n-4$ and by Huang-Ji-Yin~\cite{HJY} for $3n+1\leq N\leq 4n-7$.

The lower bound $kn+k$ of our gap is actually optimal. This can be seen by considering the expansion of $\left(\sum^k_{j=1}|z_j|^2-\sum^{n+1}_{j=k+1}|z_j|^2\right)\left(\sum^k_{j=1}|z_j|^4\right)$. The expansion is a sum of (plus or minus) norm squares of $kn+k$ linearly independent cubic monomials. Using these monomials as components we get a rational proper map from $\mathbb B^{k,n+1-k}$ to $\mathbb B^{k^2,k(n-k+1)}$ whose image does not lie in any hyperplane. Note that $\dim(\mathbb B^{k,n+1-k})=n$ and $\dim(\mathbb B^{k^2,k(n-k+1)})=kn+k-1$.
\end{proof}

\section{Proof of Lemma~\ref{binom}}\label{last section}

We will prove by induction and first show that the lemma is true for $m=1$ or $k=1$.
Suppose $k=1$ and $A+B=\binom{m+1}{1}-1=m$. If $A=0$ and $B=m$, then $A^{-<m>}+B_{<1>}=0+(m-1)=m-1=\binom{m}{1}-1$. If $1\leq A\leq m$, then
$
A=\binom{m}{m}+\binom{m-1}{m-1}+\cdots+\binom{m-A+1}{m-A+1}.
$
Hence, $A^{-<m>}+B_{<1>}=A+B-1=\binom{m}{1}-1$. 
Suppose $m=1$ and $A+B=\binom{1+k}{k}-1$. Then $B\leq k$ and hence $B_{<k>}=0$. Thus, $A^{-<1>}+B_{<k>}=0+0=\binom{k}{k}-1$. 

Suppose now $A+B=\binom{m+k}{k}-1=\binom{m+k}{m}-1=\binom{m+k-1}{m}+\binom{m+k-1}{k}-1$. Since $A,B$ are integers, we have either $A\geq\binom{m+k-1}{m}$ or $B\geq\binom{m+k-1}{k}$. 


If  $A\geq \binom{m+k-1}{m}$, then the $m$-th Macaulay's representation of $A$ is of the form 
$A=\binom{m+k-1}{m}+\binom{a_{m-1}}{m-1}+\cdots+\binom{a_\delta}{\delta}$. Since $(A-\binom{m+k-1}{m})+B=\binom{m+k-1}{k}-1$, by the induction hypothesis we get
{\footnotesize
\begin{eqnarray*}
\left(A-\binom{m+k-1}{m}\right)^{-<m-1>}+B_{<k>}&=&\binom{m+k-2}{k}-1\\
\Rightarrow\,\,\,\,\, \binom{a_{m-1}-1}{m-2}+\cdots+\binom{a_\delta-1}{\delta-1}+B_{<k>}&=&\binom{m+k-2}{k}-1\\
\Rightarrow\,\,\,\,\, \binom{m+k-2}{m-1}+\binom{a_{m-1}-1}{m-2}+\cdots+\binom{a_\delta-1}{\delta-1}+B_{<k>}&=&\binom{m+k-2}{m-1}+\binom{m+k-2}{k}-1\\
\Rightarrow\,\,\,\,\, A^{-<m>}+B_{<k>}&=&\binom{m+k-1}{k}-1
\end{eqnarray*}
}

If  $B\geq \binom{m+k-1}{k}$, then the $k$-th Macaulay's representation of $B$ is of the form 
$B=\binom{m+k-1}{k}+\binom{b_{k-1}}{k-1}+\cdots+\binom{b_\epsilon}{\epsilon}$. Since $A+(B-\binom{m+k-1}{k})=\binom{m+k-1}{k-1}-1$, by the induction hypothesis we get
{\footnotesize
\begin{eqnarray*}
A^{-<m>}+\left(B-\binom{m+k-1}{k}\right)_{<k-1>}&=&\binom{m+k-2}{k-1}-1\\
\Rightarrow\,\,\,\,\, A^{-<m>}+\binom{b_{k-1}-1}{k-1}+\cdots+\binom{b_\epsilon-1}{\epsilon}&=&\binom{m+k-2}{k-1}-1\\
\Rightarrow\,\,\,\,\, \binom{m+k-2}{k}+A^{-<m>}+\binom{b_{k-1}-1}{k-1}+\cdots+\binom{b_\epsilon-1}{\epsilon}&=&\binom{m+k-2}{k}+\binom{m+k-2}{k-1}-1\\
\Rightarrow\,\,\,\,\, A^{-<m>}+B_{<k>}&=&\binom{m+k-1}{k}-1
\end{eqnarray*}
}

\noindent{\bf Acknowledgements.} The authors would like to thank Prof. Xiaojun Huang, Ming Xiao, Wanke Yin and Yuan Yuan for their comments on the first version. The first author was partially supported by Institute of Marine Equipment, Shanghai Jiao Tong University and
Shanghai Science and Technology Plan projects (No. 21JC1401900). The second author was partially supported by Science and Technology Commission of Shanghai Municipality (STCSM) (No. 13dz2260400).


\begin{thebibliography}{999}






\bibitem[Fa1]{faran} Faran, J.:  The linearity of proper holomorphic maps between balls in the low codimension case. {\it J. Diff. Geom.} {\bf 24} (1986), 15-17.





\bibitem[Gr]{Gr} Green, M.: Restrictions of linear series to hyperplanes, and some results of Macaulay and Gotzmann. {\it Algebraic Curves and Projective Geometry} (1988), Trento Lecture Notes in Math., Vol. 1389, Springer, pp. 76-86

\bibitem[GLV]{GLV} Grundmeier, D., Lebl, J., Vivas, L.: Bounding the rank of Hermitian forms and rigidity for CR mappings of hyperquadrics. {\it Math. Ann.} {\bf 358} (2014), 1059-1089.


\bibitem[GN]{GN} Gao, Y., Ng, S.-C.: Local orthogonal maps and rigidity of holomorphic maps between real hyperquadrics. \textit{arXiv.} {\bf https://arxiv.org/abs/2110.04046}



\bibitem[HJX]{hjx} Huang, X., Ji, S., Xu, D.: A new gap phenomenon for proper holomorphic mappings from $B^n$ into $B^N$. {\it Math. Res. Lett.} {\bf 13}(4) (2006), 515-529.

\bibitem[HJY]{HJY} Huang, X., Ji, S., Yin, W.: On the third gap for proper holomorphic maps between balls. {\it Math. Ann.} {\bf 358} (2014), 115-142.

\bibitem[HJY2]{HJY2} Huang, X., Ji, S., Yin, W.: Recent Progress on Two Problems in Several Complex Variables. {\it ICCM.} Vol.I, (2007)  563-575.

\bibitem[Ma]{Ma} Macaulay, F. S.:: Some properties of enumeration in the theory of modular systems, {\it Proc. London Math. Soc.} {\bf 26} (1927), 531–555.






\bibitem[NZ]{NZ} Ng, S.-C., Zhu, Y.: Rigidity of proper holomorphic maps among generalized balls with Levi-degenerate boundaries. \textit{J. Geom. Anal.} {\bf 31} (2021), 11702–11713.





\end{thebibliography}
\end{document}